\documentclass[11pt]{article}
\usepackage[T1]{fontenc}
\usepackage{lmodern,amsmath,amsthm,amsfonts,amssymb,graphicx,float,microtype,thmtools,underscore,mathtools,anyfontsize} 
\usepackage[usenames,dvipsnames,svgnames,table]{xcolor}
\usepackage[unicode=true]{hyperref}
\usepackage{comment}
\usepackage{soul}
\usepackage{gensymb}
\hypersetup{ 
	colorlinks,
	linkcolor={black},
	citecolor={black},
	urlcolor={blue!60!black},
	pdftitle={Notes},
	pdfauthor={Vera Chekan and Torsten Ueckerdt}} 
\usepackage[noabbrev,capitalise]{cleveref}
\crefname{lem}{Lemma}{Lemmas}
\crefname{thm}{Theorem}{Theorems}
\crefname{cor}{Corollary}{Corollaries}
\crefname{prop}{Proposition}{Propositions}
\crefname{conj}{Conjecture}{Conjectures}
\crefname{rem}{Remark}{Remarks}
\crefname{openproblem}{Open Problem}{Open Problems}
\crefformat{equation}{(#2#1#3)}
\Crefformat{equation}{Equation #2(#1)#3}

\usepackage[shortlabels]{enumitem}
\setlist[itemize]{topsep=0ex,itemsep=0ex,parsep=0.4ex}
\setlist[enumerate]{topsep=0ex,itemsep=0ex,parsep=0.4ex}

\newcommand{\HH}{\mathcal{H}}
\newcommand{\EE}{\mathcal{E}}
\newcommand{\CR}{\mathcal{R}}
\newcommand{\RR}{\mathbb{R}}
\newcommand{\NN}{\mathbb{N}}

\newcommand{\pp}{\operatorname{path}}

\usepackage[longnamesfirst,numbers,sort&compress]{natbib}
\makeatletter
\def\NAT@spacechar{~}
\makeatother
\usepackage[bmargin=25mm,tmargin=25mm,lmargin=35mm,rmargin=35mm]{geometry}
\allowdisplaybreaks

\renewcommand{\geq}{\geqslant}
\renewcommand{\leq}{\leqslant}
\renewcommand{\thefootnote}{\fnsymbol{footnote}}
\allowdisplaybreaks
\theoremstyle{plain}
\newtheorem{thm}{Theorem}
\newtheorem{lem}[thm]{Lemma}
\newtheorem{cor}[thm]{Corollary}

\theoremstyle{definition}

\newtheorem{quest}[thm]{Question}
\newtheorem{definition}[thm]{Definition}

\theoremstyle{remark}
\newtheorem{rem}[thm]{Remark}

\begin{document}
	
\author{Vera Chekan\,\footnotemark[1] \qquad Torsten Ueckerdt\,\footnotemark[2] }
	
\footnotetext[1]{Humboldt-Universität zu Berlin, Germany
(\texttt{Vera.Chekan@informatik.hu-berlin.de})}

\footnotetext[2]{Institute of Theoretical Informatics, Karlsruhe Institute of Technology, Germany (\texttt{torsten.ueckerdt@kit.edu}).}
	
\sloppy
	
\title{\textbf{Polychromatic Colorings of Unions of Geometric Hypergraphs}}
	
\maketitle
	
\begin{abstract}
 We consider the polychromatic 
 coloring problems for unions of two or more geometric hypergraphs on the same vertex sets of points in the plane.
 We show, inter alia, that the union of bottomless rectangles and horizontal strips does in general not allow for polychromatic colorings.
 This strengthens the corresponding result of Chen, Pach, Szegedy, and Tardos [Random Struct. Algorithms, 34:11-23, 2009] for axis-aligned rectangles, and gives the first explicit (not randomized) construction of non-$2$-colorable hypergraphs defined by axis-parallel rectangles of arbitrarily large uniformity.
\end{abstract}
	
\renewcommand{\thefootnote}{\arabic{footnote}}

\section{Introduction}

A \emph{range capturing hypergraph} is a geometric hypergraph $\HH(V,{\CR})$ defined by a finite point set $V \subset \RR^2$ in the plane and a family $\CR$ of subsets of $\RR^2$, called \emph{ranges}.
Possible ranges are for example the family $\CR$ of all axis-aligned rectangles, all horizontal strips, or all translates of the first (north-east) quadrant.
Given the points $V$ and ranges $\CR$, the hypergraph $\HH(V,\CR) = (V,\EE)$ has $V$ as its vertex set and a subset $E \subset V$ forms a hyperedge $E\in \EE$ whenever there exists a range $R\in \CR$ with $E = V \cap R$.
That is, we have points in the plane, and a subset of points forms a hyperedge whenever these vertices and no other vertices are captured by a range.

For a positive integer $m$, we are then interested in the $m$-uniform subhypergraph $\HH(V,\CR,m)$ given by all hyperedges of size exactly $m$.
In particular, we investigate polychromatic vertex colorings $c\colon V \to [k]$ in $k$ colors of $\HH(V,\CR,m)$ for different families of ranges $\CR$, different values of $m$.
    A vertex coloring is \emph{proper} if every hyperedge contains at least two vertices of different colors.
    %
    A vertex coloring is \emph{polychromatic} if every hyperedge contains at least one vertex of each color.
    %
Note that a $2$-coloring is proper if and only if it is polychromatic.
This case is sometimes called property B for the hypergraph in the literature.
Further, a $2$-uniform hypergraph (hence, graph) admits a polychromatic 2-coloring if and only if it is bipartite.

In this paper, we mostly focus on polychromatic $k$-colorings for range capturing hypergraphs with given range family $\CR$.
In particular, we investigate the following question.

\begin{quest}\label{quest:polychromatic}
 Given $\CR$ and $k$, what is the smallest $m = m(k)$ such that for every finite point set $V \subset \RR^2$ the hypergraph $\HH(V,\CR,m)$ admits a polychromatic $k$-coloring?
\end{quest}

Of course, $m(k) \geq k$, while $m(k)=\infty$ is also possible.
Indeed, for all range families considered here, it holds that $m(2) \leq m(3) \leq \cdots$ and we either show that $m(k) < \infty$ for every $k \geq 1$ or already $m(2) = \infty$.
Note that in the latter case, there are range capturing hypergraphs that are not properly $2$-colorable, even for arbitrarily large uniformity $m$.

%
%

The motivation for studying polychromatic 
colorings of such geometric hypergraphs comes from questions of cover-decomposability and conflict-free colorings.
The interested reader is refered to the survey article~\cite{PPT13} and the excellent website~\cite{Zoo} which contains numerous references.

\subsection{Related Work}

There is a rich literature on range capturing hypergraphs, their polychromatic colorings, and answers to \cref{quest:polychromatic}.
Let us list the positive results (meaning $m(k)<\infty$ for all $k$) that are relevant here, whilst defining the respective families of ranges. 

\begin{itemize}
    \item For halfplanes $\CR = \{ \{ (x,y) \in \RR^2 \mid ax + by \geq 1\} \mid a,b \in \RR\}$ it is known that $m(k) = 2k-1$~\cite{SY12}.
    
    \item For south-west quadrants $\CR = \{\{(x,y) \in \RR^2 \mid x \leq a \text{ and } y \leq b\} \mid a,b \in \RR\}$ it is easy to prove that $m(k) = k$, see e.g.~\cite{KP15}.
    
    
    \item For axis-aligned strips $\CR = \{\{(x,y) \in \RR^2 \mid a_1 \leq x \leq a_2\} \mid a_1,a_2 \in \RR\} \cup \{\{(x,y) \in \RR^2 \mid a_1 \leq y \leq a_2\} \mid a_1,a_2 \in \RR\}$ it is known that $m(k) \leq 2k-1$~\cite{ACCIKLSST11}.
    
    \item For bottomless rectangles $\CR = \{\{(x,y) \in \RR^2 \mid a_1 \leq x \leq a_2 \text{ and } y \leq b\} \mid a_1,a_2,b \in \RR\}$ it is known that $1.67k \leq m(k) \leq 3k-2$~\cite{ACCCHHKLLMRU13}.
    
    \item For axis-aligned squares $\CR = \{ \{(x,y) \in \RR^2 \mid a \leq x \leq a+s \text{ and } b \leq y \leq b+s\} \mid a,b,s \in \RR\}$ it is known that $m(k) \leq O(k^{8.75})$~\cite{AKV17}.
\end{itemize}

On the contrary, let us also list the negative results (meaning $m(k) = \infty$ for some $k$) that are relevant here.
In all cases, it is shown that already $m(2)=\infty$, meaning that there is a sequence $(\HH_m)_{m \geq 1}$ of hypergraphs such that for each $m \geq 1$ we have $\HH_m = \HH(V_m,\CR,m)$ for some point set $V_m$ (in this case we say that $\HH_m$ can be \emph{realized} with $\CR$) and $\HH_m$ admits no polychromatic $2$-coloring.
One such sequence are the \emph{$m$-ary tree hypergraphs}, defined on the vertices of a complete $m$-ary tree of depth $m$, where for each non-leaf vertex, its $m$ children form a hyperedge, and for each leaf vertex, its $m$ ancestors form a hyperedge (introduced by \citet{PTT07}).
A second such sequence is due to \citet{P13} (published in~\cite{PP16}), for which we do not repeat the formal definition here and simply refer to them as the \emph{2-size hypergraphs} as their inductive construction involves hyperedges of two possibly different sizes.

\begin{itemize}
    \item For strips $\CR = \{ \{(x,y) \in \RR^2 \mid 1 \leq ax + by \leq c\} \mid a,b,c \in \RR\}$ it is known that $m(2)=\infty$ as every $m$-ary tree hypergraph can be represented with strips~\cite{PTT07}.
    
    \item For unit disks $\CR = \{ \{(x,y) \in \RR^2 \mid (x-a)^2 + (y-b)^2 \leq 1\} \mid a,b \in \RR \}$ it is known that $m(2)=\infty$ as every 2-size hypergraph can be represented with unit disks~\cite{PP16}.
\end{itemize}

Finally, for axis-aligned rectangles $\CR = \{ \{(x,y) \in \RR^2 \mid a_1 \leq x \leq a_2 \text{ and } b_1 \leq y \leq b_2\} \mid a_1,a_2,b_1,b_2 \in \RR\}$ it is also known that $m(2) = \infty$.
See \cref{thm:rectangles} below.
However, the only known proof of \cref{thm:rectangles} is a probabilistic argument and we do not have any explicit construction of a sequence $(\HH_m)_{m \geq 1}$ of $m$-uniform hypergraphs defined by axis-aligned rectangles that admit no polychromatic $2$-coloring.

\begin{thm}[Chen et al.~\cite{CPST09}]\label{thm:rectangles}{\ \\}
 For the family $\CR$ of all axis-aligned rectangles it holds that $m(2) = \infty$.
 
 That is, for every $m \geq 1$ there exists a finite point set $V \subset \RR^2$ such that for every $2$-coloring of $V$ some axis-aligned rectangle contains $m$ points of $V$, all of the same color.
\end{thm}

%

\subsection{Our Results}

In this paper we consider range families $\CR = \CR_1 \cup \CR_2$ that are the union of two range families $\CR_1$, $\CR_2$.
The corresponding hypergraph $\HH(V,\CR,m)$ is then the union of the hypergraphs $\HH(V,\CR_1,m)$ and $\HH(V,\CR_2,m)$ on the same vertex set $V \subset \RR^2$.
Clearly, if $\HH(V,\CR,m)$ is polychromatic $k$-colorable, then so are $\HH(V,\CR_1,m)$ and $\HH(V,\CR_2,m)$.
But the converse is not necessarily true and this shall be the subject of our investigations.

In \cref{sec:general} we show that if $\CR_1$ and $\CR_2$ admit so-called hitting $k$-cliques, then we can conclude that $m(k) < \infty$ for $\CR = \CR_1 \cup \CR_2$.
This is for example the case for all horizontal (respectively vertical) strips, but already fails for all south-west quadrants.
In \cref{sec:unbounded-rectangles} we then consider all possible families of \emph{unbounded} axis-aligned rectangles, such as axis-aligned strips, all four types of quadrants, or bottomless rectangles.
We determine exactly for which subset of those, when taking $\CR$ as their union, it holds that $m(k) < \infty$.

In particular, we show in \cref{subsec:negative} that $m(k) = \infty$ for all $k \geq 2$ when $\CR = \CR_1 \cup \CR_2$ is the union of $\CR_1$ all bottomless rectangles and $\CR_2$ all horizontal strips.
Our proof gives a new sequence $(\HH_m)_{m \geq 1}$ of $m$-uniform hypergraphs that admit a geometric realization for simple ranges, but do not admit any polychromatic $2$-coloring.
On the positive side, we show in \cref{subsec:positive} that (up to symmetry) all other subsets of unbounded axis-aligned rectangles (excluding the above pair) admit polychromatic $k$-colorings for every $k$.
Here, our proof relies on so-called shallow hitting sets and in particular a variant in which a subset of $V$ hits every hyperedge defined by $\CR_1$ at least once and every hyperedge defined by $\CR_1 \cup \CR_2$ at most a constant (usually $2$ or $3$) number of times.

\paragraph{Assumptions and Notation.}
Before we start, let us briefly mention some convenient facts that are usually assumed, and which we also assume throughout our paper:
Whenever a range family $\CR$ is given, we only consider points set $V$ that are in general position with respect to $\CR$.
For us, this means that the points in $V$ have pairwise different $x$-coordinates, pairwise different $y$-coordinates, and also pairwise different sums of $x$- and $y$-coordinates.
Secondly, all range families $\CR$ that we consider here are \emph{shrinkable}, meaning that whenever a set $X \subseteq V$ of $i$ points is captured by a range in $\CR$, then also some subset of $i-1$ points of $X$ is captured by a range in $\CR$.
This means that for every polychromatic $k$-coloring of $\HH(V,\CR,m)$, every range in $\CR$ capturing $m$ or more points of~$V$, contains at least one point of each color, thus giving $m(2) \leq m(3) \leq \cdots$ as mentioned above. 
Finally, for every set $X \subseteq V$ captured by a range in $\CR$, we implicitly associate to $X$ one arbitrary but fixed such range $R \in \CR$ with $V \cap R = X$. 
In particular, we shall sometimes consider \emph{the} range $R$ for a given hyperedge $E$ of $\HH(V,\CR,m)$.

\section{Polychromatic Colorings for Two Range Families}\label{sec:general}

Let $\CR_1,\CR_2$ be two families of ranges, for each of which it is known that $m(k) < \infty$ for any $k \geq 1$.
We seek to investigate whether also for $\CR = \CR_1 \cup \CR_2$ we have $m(k) < \infty$.
First, we identify a simple sufficient condition.

For fixed $k,m,\CR$, we say that we have \emph{hitting $k$-cliques} if the following holds.
For every $V \subset \RR^2$ there exist pairwise disjoint $k$-subsets of $V$ such that every hyperedge of $\HH(V,\CR,m)$ fully contains at least one such $k$-subset.
Clearly, if we have hitting $k$-cliques, then $m(k) \leq m$ since we can simply use all colors $1,\ldots,k$ on each such $k$-subset (and color any remaining vertex arbitrarily).

\begin{thm}\label{thm:hitting-cliques}
 For fixed $k,m$, suppose that we have hitting $k$-cliques with respect to $\CR_1$ and hitting $k$-cliques with respect to $\CR_2$.
 Then for $\CR = \CR_1 \cup \CR_2$ it holds that $m(k) \leq m$.
\end{thm}

\begin{proof}
 Consider $V \subset \RR^2$, a set $A$ of hitting $k$-cliques of $\HH(V,\CR_1,m)$ and a set $B$ of hitting $k$-cliques of $\HH(V,\CR_2,m)$.
 Let $G = (A\cup B, \EE)$ be the hypergraph with one hyperedge $E_v \in \EE$ for each vertex $v \in V$ such that for~$X \in A \cup B$, we have $X \in E_v$ if and only if $v \in X$.
 Removing the empty hyperedges from $G$, we have that $G$ is a bipartite multigraph together with some additional loops.
 As every vertex of $G$ has degree exactly $k$ (loops contribute only once to a vertex degree), it follows that $G$ admits a proper $k$-edge-coloring~\cite{B73}, which interpreted as a $k$-coloring of $V$ gives the result since every $k$-clique gets all $k$ colors.
\end{proof}

As a corollary, we easily reprove the upper bound on $m(k)$ from~\cite{ACCIKLSST11} when $\CR$ consists of all axis-parallel strips (or more generally, all cones with apex at one of two fixed points).

\begin{cor}\label{cor:aligned-strips}
 For $\CR$ the family of all axis-aligned strips and any $k$ we have $m(k) \leq 2k-1$.
\end{cor}
\begin{proof}
 For $m = 2k-1$, the family $\CR_1 = \{ \{(x,y) \in \RR^2 \mid a_1 \leq x \leq a_2\} \mid a_1,a_2 \in \RR, a_1 < a_2\}$ of all vertical strips has hitting $k$-cliques by grouping always $k$ points with consecutive $x$-coordinates, leaving out the last up to $k-1$ points.
 Symmetrically, we have hitting $k$-cliques for the family $\CR_2$ of all horizontal strips.
 As $\CR = \CR_1 \cup \CR_2$, we have $m(k) \leq m = 2k-1$ by \cref{thm:hitting-cliques}.
\end{proof}

Somewhat unfortunately, hitting $k$-cliques appear to be very rare.
Already for the range family $\CR$ of all south-west quadrants, for which one can easily show that $m(k)=k$, we do not even have hitting $2$-cliques for any $m$.
This will follow from the following result, which will also be useful later.

\begin{lem}\label{lem:tree-with-quadrants}
 Let $T$ be a rooted tree, and $\HH(T)$ be the hypergraph on $V(T)$ where for each leaf vertex its ancestors (including itself) form a hyperedge.
 Then $\HH(T)$ can be realized with the family $\CR$ of all south-west quadrants.
 
 Moreover, the root is the bottommost and leftmost point and the children of each non-leaf vertex lie on a diagonal line of slope $-1$. 
\end{lem}
\begin{proof}
 We do induction on the height of $T$, with height $1$ being a trivial case of a single vertex.
 For height at least $2$, remove the root $r$ from $T$ to obtain new trees $T_1,\ldots,T_p$, each of smaller height and rooted at a child of $r$.
 By induction, there are point sets in the plane for each $\HH(T_i)$, $i=1,\ldots,p$, with each respective root being bottommost and leftmost. We scale each of these points sets uniformly until the bounding box of each of them has width as well as height less than $1$. For every $i \in [p]$, we put the point set for $\HH(T_i)$ into the plane so that the root of $T_i$ has the coordinate $(i, p - i)$. Finally, we place $r$ in the origin. This gives the desired representation. 
\end{proof}

\begin{cor}\label{cor:no-hitting-quadrants}
 For $k=2$, $\CR$ the family of all south-west quadrants, and any $m \geq 2$, we do not have hitting $k$-cliques.
\end{cor}
\begin{proof}
 Take the rooted complete $m$-ary tree $T_m$ of depth $m$, for which $\HH(T_m)$ is realizable with south-west quadrants by \cref{lem:tree-with-quadrants}.
 By induction on $m$, we show that $\HH(T_m)$ does not have hitting $2$-cliques. This is trivial for $m=2$.
 Otherwise, any collection of disjoint $2$-subsets either avoids the root $r$, or pairs $r$ with a vertex in one of the $m \geq 2$ subtrees of $T_m$ below $r$.
 In any case, there is a subtree $T$ below $r$, none of whose vertices is paired with $r$ and hence, there exist hitting 2-cliques of $\HH(T)$. Note that $T$ is a complete $m$-ary tree of depth $m-1$ so it contains $T_{m-1}$ as a subtree. But then $T_{m-1}$ admits hitting $2$-cliques too. A contradiction to induction hypothesis.

\end{proof}

\begin{cor}\label{cor:no-hitting-halfplanes}
 For $k=2$, $\CR$ the family of all halfplanes, and any $m \geq 2$, we do not have hitting $k$-cliques.
\end{cor}
\begin{proof}
 By a result of \citet{MP92}, every range capturing hypergraph for south-west quadrants can also be realized by halfplanes and the result follows from~\cref{cor:no-hitting-quadrants}.
\end{proof}

To summarize, parallel strips have hitting $k$-cliques, but quadrants do not.
Hence, we can not apply \cref{thm:hitting-cliques} to conclude that maybe we have $m(k) < \infty$ when we consider~$\CR$ to be the union of all quadrants of one direction and all parallel strips of one direction.
In \cref{sec:unbounded-rectangles} we shall prove that indeed $m(k) < \infty$ for the union of all quadrants and strips, however only provided that the strips are axis-aligned.
In fact, if they are not, this is not necessarily true.

\begin{cor}\label{cor:quadrants-and-diagonal}
 Let $\CR_1$ be the family of all south-west quadrants and $\CR_2 = \{ \{(x,y) \in \RR^2 \mid a \leq x+y \leq b\} \mid a,b \in \RR\}$ be the family of all diagonal strips of slope $-1$.
 
 Then for $\CR = \CR_1 \cup \CR_2$ we have $m(2) = \infty$.
\end{cor}
\begin{proof}
 Given $m$, consider the rooted complete $m$-ary tree $T_m$ of depth $m$.
 By \cref{lem:tree-with-quadrants},~$V(T_m)$ can be placed in the plane such that for each leaf vertex, its $m$ ancestors (including itself) are captured by a south-west quadrant, and for each non-leaf vertex, its $m$ children are captured by a diagonal strip.
 Using a slight pertubation of the points, we can ensure that the lines of slope $-1$ guaranteed by \cref{lem:tree-with-quadrants} are pairwise distinct.
 Hence, every $m$-ary tree hypergraph $\HH_m$ can be represented with $\CR$. By \cite{PTT07} the hypergraph $\HH_m$ does not admit a polychomatic $2$-coloring for every $m$, which gives the result.
\end{proof}

\section{Families of Unbounded Rectangles}
\label{sec:unbounded-rectangles}

In this section we consider the following range families of unbounded rectangles:
\begin{itemize}
    \item all (axis-aligned) south-west quadrants $\CR_{\rm SW}$,
    
    \item similarly all south-east $\CR_{\rm SE}$, north-east $\CR_{\rm NE}$, north-west $\CR_{\rm NW}$ quadrants,
    
    \item all horizontal strips $\CR_{\rm HS}$, vertical strips $\CR_{\rm VS}$, diagonal strips $\CR_{\rm DS}$ of slope $-1$,
    
    \item all bottomless rectangles $\CR_{\rm BL}$, and finally all topless rectangles $\CR_{\rm TL} = \{ \{(x,y) \in \RR^2 \mid a_1 \leq x \leq a_2, b \leq y\} \mid a_1,a_2,b \in \RR\}$.
\end{itemize}
Observe that if a point set is captured by a south-east  quadrant $Q$, then it is also captured by a bottomless rectangle having the same top and left sides as $Q$ and whose right side lies to the right of every point in the vertex set. Analogous statements hold for other quadrants and vertical strips. Further, note that each of the above range families, except the diagonal strips $\CR_{\rm DS}$, is a special case of the family of all axis-aligned rectangles.
Recall that for the family of \emph{all} axis-aligned rectangles, it is known that $m(2) = \infty$~\cite{CPST09}.
Here we are interested in the maximal subsets of $\{ \CR_{\rm SW}, \CR_{\rm SE}, \CR_{\rm NE}, \CR_{\rm NW}, \CR_{\rm HS}, \CR_{\rm VS}, \CR_{\rm BL}, \CR_{\rm TL}\}$ so that for the union $\CR$ of all these ranges it still holds that $m(k) < \infty$ for all $k$.
In fact, we shall show that for $\CR = \CR_{\rm BL} \cup \CR_{\rm HS}$ we have $m(2) = \infty$, strengthening the result for axis-aligned rectangles~\cite{CPST09}.
On the other hand, for $\CR = \CR_{\rm SW} \cup \CR_{\rm SE} \cup \CR_{\rm NE} \cup \CR_{\rm NW} \cup \CR_{\rm HS} \cup \CR_{\rm VS}$, i.e., the union of all quadrants and axis-aligned strips, we have $m(k) < \infty$ for all $k$, strengthening the results for south-west quadrants~\cite{KP15} and axis-aligned strips~\cite{ACCIKLSST11}.
Secondly, for $\CR = \CR_{\rm BL} \cup \CR_{\rm TL}$, i.e., the union of bottomless and topless rectangles (which also contains all quadrants and all vertical strips), we again have $m(k) < \infty$ for all $k$, thus strengthening the result for bottomless rectangles~\cite{ACCCHHKLLMRU13}.
Using symmetries, this covers all cases of the considered unbounded axis-aligned rectangles.
We complement our results by also considering the diagonal strips $\CR_{\rm DS}$ and recall that we already know by \cref{cor:quadrants-and-diagonal} that for $\CR = \CR_{\rm DS} \cup \CR_{\rm SW}$ we have $m(2) = \infty$.

\subsection{The Case with no Polychromatic Coloring: Bottomless Rectangles and Horizontal Strips}\label{subsec:negative}



\begin{definition}\label{def:non-2-colorable}
    For every $m \in \NN$, the $m$-uniform hypergraph $\HH_m$ is defined as follows. First we define a rooted forest $F_m$ consisting of $m^m$ trees whose vertices are partitioned into a set of the so-called \emph{stages}. The vertices of a stage $S$ will be totally ordered and we denote this ordering by $<_S$. All vertices of a stage $S$ will have the same distance to the root of the corresponding tree, we refer to this distance as the \emph{level} of $S$. Every stage on level $j \in \{0, \dots, m-1\}$ will consist of $m^{m-j}$ vertices.
    
    To define the rooted forest $F_m$ and the stages, we start with $m^m$ roots, one for each tree in $F_m$.
    They build the unique stage on level $0$ and they are ordered in an arbitrary but fixed way. After that, for every already defined stage $S$ on level $j < m - 1$ and every subset $S' \in {S \choose m^{m-j-1}}$, we add a new stage $T(S')$ on level $j+1$ consisting of $m^{m-j-1}$ new vertices so that every vertex in $S'$ gets exactly one child from $T(S')$ and the vertices of $T(S')$ are ordered by $<_{T(S')}$ as their parents by $<_S$. Informally speaking, every vertex in $S$ gets a child for every $(m^{m-j-1})$-subset of $S$ in which it occurs. 
    
    Now we can define the hypergraph $\HH_m = (V, \EE)$. The vertex set $V$ is exactly $V(F_m)$. There are two types of hyperedges. First, stage-hyperedges $\EE_s$: for every stage $S$, each $m$ consecutive vertices in $<_S$ constitute a stage-hyperedge. Second, the path-hyperedges $\EE_p$: every root-to-leaf path in $F_m$ forms a path-hyperedge. Then, the set of hyperedges is defined as $\EE = \EE_s \cup \EE_p$. Note that $\HH_m = (V,\EE)$ is indeed $m$-uniform. For a vertex $v$, let $\operatorname{root}(v)$ denote the root of the tree in $F_m$ containing $v$, and $\pp(v) \subset V$ denote the vertices on the path from $v$ to $\operatorname{root}(v)$ in $F_m$.
\end{definition}

\begin{thm}
 For every $m \in \NN$ the $m$-uniform hypergraph $\HH_m = (V, \EE = \EE_s \cup \EE_p)$ admits no polychromatic coloring with 2 colors.
\end{thm}

\begin{proof}
    We show that every $2$-coloring of $V$ that makes all stage-hyperedges polychromatic necessarily produces a monochromatic path-hyperedge. Let ${\phi: V \rightarrow \{\text{red, blue}\}}$ be such a coloring. The key observation is that a stage $S$ on level $j$ (i.e., one that contains $m^{m-j}$ vertices) can be partitioned into $m^{m-j} / m = m^{m-j-1}$ disjoint stage-hyperedges and hence, it contains at least $m^{m-j-1}$ red vertices.
    
    We prove for $j \in \{0, \dots, m-1\}$ that there is a stage $S_j$ on level $j$ and a subset $B_j \subset S_j$ such that $|B_j| = m^{m-j-1}$ and for every vertex $v \in B_j$, the vertices in $\pp(v)$ are all red. For $j = 0$, the stage consisting of roots contains at least $m^{m-1}$ red roots and these vertices have the desired property. Suppose the statement holds for some $j$. Consider the stage $S_{j+1} = T(B_j)$ and a set $B_{j+1}$ of $m^{m-j-2}$ red points in it. By definition, each of these points $v$ has its parent in $B_j$ and hence, all vertices in $\pp(v)$ are red.
    So the statement holds for $j+1$, too. By induction, it holds for $j = m-1$ and hence, there is a vertex $v$ on level $m-1$ (i.e., a leaf) whose root-to-leaf path is completely red. So $\HH_m$ admits no polychromatic 2-coloring.  
\end{proof}

\begin{thm}
 For every $m \in \NN$ the $m$-uniform hypergraph $\HH_m = (V, \EE = \EE_s \cup \EE_p)$ admits a realization with bottomless rectangles and horizontal strips.
\end{thm}

\begin{proof}
    For a point $p \in \RR^2$, let $x(p)$ and $y(p)$ denote its $x$- respectively $y$-co\-or\-di\-nate.
    We call a sequence of points $p_1, \dots, p_t$ \emph{ascending} (respectively \emph{descending}) if ${x(p_1) < \dots < x(p_t)}$ and ${y(p_1) < \dots < y(p_t)}$ (respectively ${x(p_1) < \dots < x(p_t)}$ and ${y(p_1) > \dots > y(p_t)}$). Writing about the vertices of a stage $S$, we always refer to their ordering in $<_S$. 
    We shall embed each stage $S$ of $\HH_m$ into a closed horizontal strip, denoted $H_S$, in such a way that $H_S \cap H_{S'} = \emptyset$ whenever $S \neq S'$.
    Note that this way, the embedded stages are vertically ordered with some available space between any two consecutive ones.  
    
    First, we embed the roots of $F_m$, i.e., the unique stage on level $0$, as an ascending sequence in a horizontal strip for this stage. After that, we choose some stage $S$ that has already been embedded but the stages $T_1, \ldots, T_r$ containing its children not yet. In one step we embed $T_1,\ldots,T_r$ as follows. We pick a thin horizontal strip $H$ between $H_S$ and the strip above (if it exists) and within $H$ identify disjoint horizontal strips $H_1, \dots, H_r$. Then, every $T_i$ is embedded inside $H_i$ so that every vertex gets initially the same $x$-coordinate as its parent and the vertices of $T_i$ build an ascending sequence in $H_i$. After that, for every $p \in S$ we slightly shift all children of $p$ to the right so that they build a descending sequence but the ordering of $x$-coordinates relative to all other points remains unchanged.

    The arising embedding ensures the following two properties. First, every stage-hyperedge is captured by a horizontal strip. Second, for every vertex $v$, the bottomless rectangle $B(v)$ with top-right corner $v$ and $\operatorname{root}(v)$ on the left side captures exactly $\pp(v)$. The first property holds since every stage $S$ is embedded as an ascending sequence (along $<_S$) in a thin horizontal strip $H_S$ and these horizontal strips are pairwise disjoint. We prove the second property by induction on the level of the unique stage $T$ containing $v$.
    If $T$ is the stage on level $0$, then $v$ is a root and clearly $B(v)$ contains only $v$.
    Otherwise, let $w$ be the parent of $v$ and $S$ be the stage containing $w$.
    Then $B(v)$ arises from $B(w)$ by extending its top-right corner from $w$ to $v$, giving their difference $D = B(v) - B(w)$ an L-shape as illustrated in \cref{fig:stage-bottomless}.
    We claim that $D$ contains only the point $v$ from $V$, which by induction then gives $B(v) \cap V = (B(w) \cap V) \cup (D \cap V) = \pp(w) \cup \{v\} = \pp(v)$, as desired.
    
    \begin{figure}
        \centering
        \includegraphics{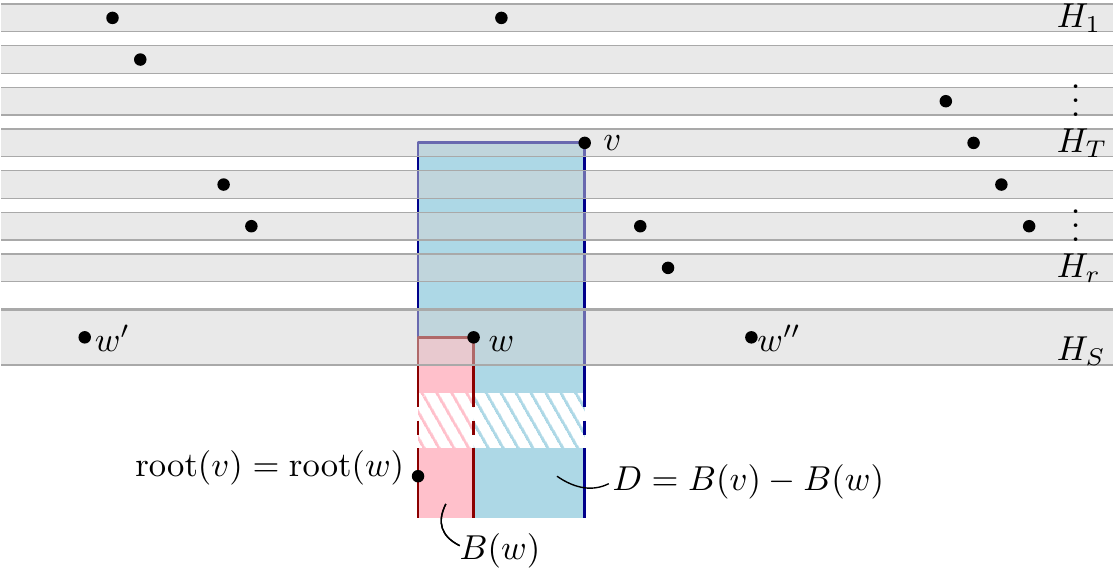}
        \caption{For stage $S$ on level $j$, embedding all stages containing the children of vertices in $S$ into thin horizontal strips slightly above the strip $H_S$ for $S$.}
        \label{fig:stage-bottomless}
    \end{figure}
    
    To see that $D \cap V = \{v\}$, first consider the step in which $v$ was embedded.
    Since $v$ lies slighty to the right of $w$ and the vertices of $S$ form an increasing sequence, no further vertex of $S$ belongs to $D$.
    All vertices lying vertically between $w$ and $v$ also lie between $H_S$ and $H_T$ so in this step all such vertices have their parent in $S$ (or they belong to $S$ but we have already excluded this case). 
    Each of these vertices lies slightly to the right of its parent. Since $B(w) \cap H_S = \{w\}$, the only vertices that could lie in $D$ above $w$ are the children of $w$. But they form a decreasing sequence so $v$ is the only such vertex. 
    Since $v$ lies slightly to the right of $w$, there is also no further point in $D$ below $w$ in this step.
    Finally, each further point (i.e., embedded after this step) is embedded above and ever so slightly to the right of its parent.
    As every vertex in $\pp(w)$ has all its children already embedded, no further points are embedded into $D$.
    So the hypergraph $\HH_m$ can be realized by bottomless rectangles and horizontal strips.
\end{proof}

\subsection{The Cases with Polychromatic Colorings}\label{subsec:positive}

First, recall the result of \citet{AKV17} that for the range family $\CR_{\rm SQ}$ of all axis-aligned squares, we have $m(k) \leq O(k^{8.75})$.
This already seals the deal for bottomless and topless rectangles.

\begin{thm}\label{thm:bottomless-topless-rectangles}
 For the family $\CR = \CR_{\rm BL} \cup \CR_{\rm TL}$ of all bottomless and topless rectangles, we have $m(k) \leq O(k^{8.75})$ for all $k$. 
\end{thm}

\begin{proof}
 Let $V$ be a finite point set and let $m$ be arbitrary. 
 For every bottomless (respectively topless) rectangle capturing a hyperedge of $\HH = \HH(V, \CR, m)$, we introduce a bottom (respectively top) side below the bottommost (respectively above the topmost) point in $V$ so that these rectangles are bounded now. After that we stretch the plane horizontally until the width of every aforementioned rectangle becomes larger than its height and obtain the point set $V'$. This stretching preserves the ordering of $x$- and $y$-coordinates of the points so that the set of hyperedges captured by $\CR$ remains the same. Finally, we pick every (now bounded) bottomless (respectively topless) rectangle capturing a hyperedge of $\HH$ and shift its bottom (respectively top) side down (respectively up) until it becomes a square. 
 Now for every hyperedge in $\HH$, there is an axis-aligned square capturing it and hence, a hyperedge in $\HH' = (V', \CR_{\rm SQ}, m)$. Thus, each polychromatic coloring of $\HH'$ yields a polychromatic coloring of $\HH$ and this concludes the proof.
\end{proof}

For the remaining cases, we utilize so-called shallow hitting sets.
For a positive integer $t$, a subset $X$ of vertices of a hypergraph $\HH$ is a \emph{$t$-shallow hitting set} if every hyperedge of $\HH$ contains at least one and at most $t$ points from $X$.
It is known for example that for $\CR$ being the family of all halfplanes, every range capturing hypergraph $\HH(V,\CR,m)$ admits a $2$-shallow hitting set~\cite{SY12}, which implies that $m(k) \leq 2k-1$ in this case.
In general, we have the following.

\begin{lem}[\citet{KP19}]\label{lem:coloring-from-shallow}{\ \\}
 Suppose that for a shrinkable range family $\CR$, every hypergraph $\HH(V,\CR,m)$ admits a $t$-shallow hitting set.
 Then $m(k) \leq (k-1)t + 1$.
\end{lem}
\begin{proof}
 Doing induction on $k$, we first observe that the claim holds for $k=1$, as in this case $(k-1)t + 1 = 1$.
 For $k \geq 2$ let $m = (k-1)t + 1$ and consider a $t$-shallow hitting set $X$ of $\HH(V,\CR,m)$.
 Then every hyperedge contains at least $m - t$ points of $V - X$.
 Since $\CR$ is shrinkable, for every hyperedge $E$ of $\HH(V, \CR, m)$, there is a hyperedge $E'$ of $\HH(V - X, \CR, m - t)$ with $E' \subseteq E$. So if $E'$ is polychromatic in some coloring, then $E$ is as well. By induction, $\HH(V-X,\CR,m-t = (k-2)t + 1)$ admits a polychromatic coloring in colors $1,\ldots,k-1$. Taking $X$ as the $k$-th color, gives a polychromatic $k$-coloring of $\HH(V,\CR,m)$.
\end{proof}

\begin{rem}
 \cref{lem:coloring-from-shallow} states that if $t$-shallow hitting sets exist (for a global constant $t$), then $m(k) = O(k)$.
 However, it is not clear whether the converse is also true, for example when $\CR$ is the family of all bottomless rectangles. 
 \citet{KP19} construct for this family range capturing hypergraphs without shallow hitting sets, but their constructed hypergraphs are not uniform.
 In fact, it follows from the proof of \cref{cor:aligned-strips} that axis-aligned strips do in fact admit $3$-shallow hitting sets, since every hyperedge of $\HH(V,\CR_{\rm HS} \cup \CR_{\rm VS})$ of size $2k-1$ or $2k$ is hit by at most three of the hitting $k$-cliques, and thus contains color $1$ of the resulting $k$-coloring at least once and at most three times. 
 To the best of our knowledge, it is open whether all $\HH(V,\CR,m)$ admit shallow hitting sets for the bottomless rectangles $\CR = \CR_{\rm BL}$.
\end{rem}

Recall that for the family $\CR_{\rm NW}$ of all north-west quadrants we have $m(k)=k$, and observe that in such a polychromatic coloring, every color class is a $1$-shallow hitting set.
Besides north-west quadrants, we want to consider other range families.
Thus, we are interested in $t$-shallow hitting sets for $\CR_{\rm NW}$ that additionally do not hit other ranges, such as axis-parallel strips or other quadrants, too often.

\begin{lem}\label{lem:2-shallow-quadrants}
 For the family $\CR_{\rm NW}$ 
 of all north-west 
 quadrants, every hypergraph $\HH(V,\CR_{\rm NW},m)$ 
 admits a $2$-shallow hitting set $X$ such that the points $x_1,\ldots,x_n$ in $X$ have decreasing 
 $x$-coordinates and decreasing 
 $y$-coordinates, and
 \begin{enumerate}[(i)]
     \item $x_1$ is the leftmost 
     point of the $m$ topmost 
     points in $V$,\label{item:x1-leftmost-top}
    
    \item the hyperedge consisting of the topmost 
    $m$ vertices of $V$ is hit exactly once, \label{item:topmost-hyperedge}
    
     \item for any two consecutive points $x_j,x_{j+1}$ in $X$, the bottomless 
     rectangle $B_j$ with top-right 
     corner $x_j$ and $x_{j+1}$ on the left 
     side satisfies $|B_j \cap V| \geq m+1$, and\label{item:2-rectangle}
     
     \item for any three consecutive points $x_j,x_{j+1},x_{j+2}$ in $X$, the axis-aligned rectangle $R_j$ with top-right 
     corner $x_j$ and bottom-left 
     corner $x_{j+2}$ satisfies $|R_j \cap V| \geq m+2$.\label{item:3-rectangle}
 \end{enumerate}
\end{lem}
\begin{proof}
 For each hyperedge of $\HH(V,\CR_{\rm NW},m)$ consider a fixed north-west quadrant capturing these $m$ points of $V$.
 These quadrants can be indexed $Q_1,\ldots,Q_\alpha$ along their apices with decreasing $x$-coordinates (and hence also $y$-coordinates).
 I.e., $Q_1$ contains the topmost $m$ points of $V$, while $Q_\alpha$ contains the leftmost $m$ points of $V$.
 See Figure~\ref{fig:all-quadrants} for an illustrative example.
 
 \begin{figure}[t]
    \centering
    \includegraphics{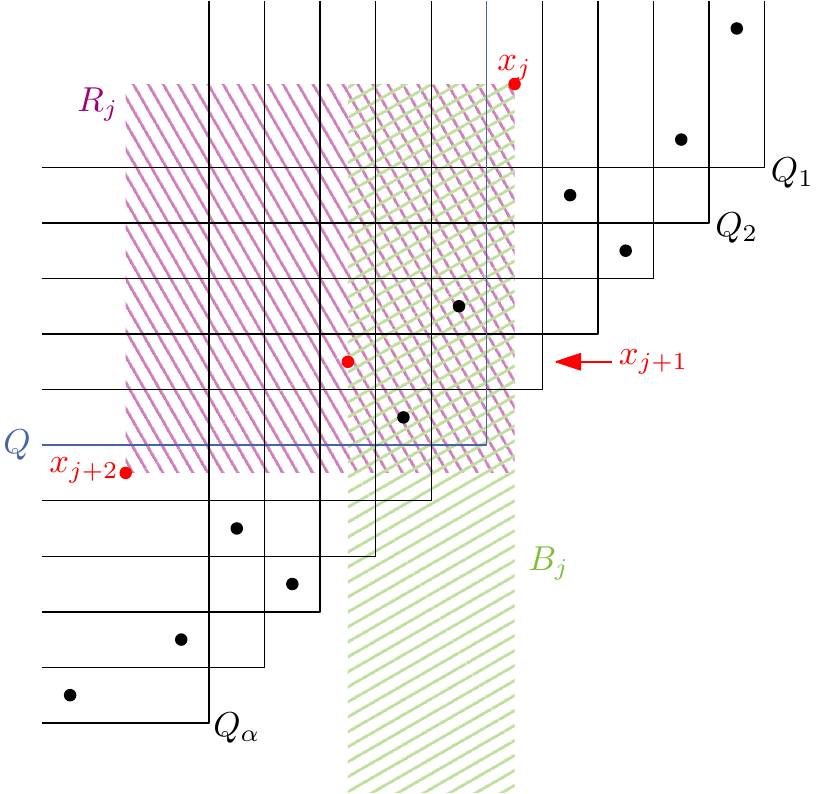}
    \caption{
        Sketch for the proof of Lemma~\ref{lem:2-shallow-quadrants} for $m = 3$. The vertices in $X$ are red.
    }
    \label{fig:all-quadrants}
\end{figure}
 
 Starting with $X = \emptyset$, we go through the north-west quadrants from $Q_1$ to $Q_\alpha$, and whenever $Q_i$ does not contain any point of $X$, we add the leftmost point of $Q_i \cap V$ to $X$.
 Label the points in $X$ by $x_1,\ldots,x_n$ in the order in which they were added to $X$.
 Along this order, the points have decreasing $x$-coordinates and decreasing $y$-coordinates.
 Clearly, $X$ is a hitting set of $\HH(V,\CR_{\rm NW},m)$ and satisfies \cref{item:x1-leftmost-top}.
 
 Since $x_1$ is the leftmost point of $Q_1$, $x_j$ does not belong to $Q_1$ for every $j \in \{2, \dots, n\}$. 
 Since $Q_1$ contains exactly the topmost $m$ vertices, 
 the corresponding hyperedge is hit exactly once. This proves \cref{item:topmost-hyperedge}.
 
 For any two consecutive points $x_j,x_{j+1}$ in $X$, consider the bottomless rectangle $B_j$ with top-right corner $x_j$ and $x_{j+1}$ on the left side.
 Then $|B_j \cap V| \geq m+1$ as $B_j$ contains $x_j$ and all points of the north-west quadrant $Q$ for which we added $x_{j+1}$ to $X$.
 This proves \cref{item:2-rectangle}.
 
 Moreover, every point in $Q \cap V$ lies above $x_{j+2}$ (if it exists), as $x_{j+1}$ is leftmost in $Q$.
 Thus, the axis-aligned rectangle $R_j$ with top-right corner $x_j$ and bottom-left corner $x_{j+2}$ contains $x_j$, all the $m$ points in $Q \cap V$, and $x_{j+2}$.
 This proves \cref{item:3-rectangle} and also implies that $X$ is $2$-shallow.
\end{proof}

Let $E_t(V, m)$ (respectively $E_b(V, m)$) denote the set of $m$ topmost (respectively bottommost) points in $V$. For symmetry reasons, statements analogous to the above lemma hold for other types of quadrants so we obtain the following corollary:
\begin{cor}\label{cor:shallow-hitting-sets}
For a point set $V$ and $m \in \NN$, a hypergraph 
	\[
		\HH(V, \CR_{\rm{NW}}, m) \text{ } / \text{ } \HH(V, \CR_{\rm{NE}}, m) \text{ } / \text{ } \HH(V, \CR_{\rm{SW}}, m) \text{ } / \text{ } \HH(V, \CR_{\rm{SE}}, m)
	\] 
	admits a hitting set 
	\[
		S_{\rm{NW}} \text{ } / \text{ } S_{\rm{NE}} \text{ } / \text{ } S_{\rm{SW}} \text{ } / \text{ } S_{\rm{SE}} \text{ }
	\]
	such that:
	\begin{enumerate}
		\item It satisfies the properties of (the symmetrical version of) \cref{lem:2-shallow-quadrants}.	
		\item Every hyperedge $E \neq E_t(V, m)$ of $\HH(V, \CR_{\rm{NW}}, m)$ is hit by $S_{\rm{NW}}$~/ $S_{\rm{NE}}$~/ $S_{\rm{SW}}$~/ $S_{\rm{SE}}$ at most twice~/ not once~/ once~/ once. \label{item:nw}
		\item Every hyperedge $E \neq E_t(V, m)$ of $\HH(V, \CR_{\rm{NE}}, m)$ is hit by $S_{\rm{NW}}$~/ $S_{\rm{NE}}$~/ $S_{\rm{SW}}$~/ $S_{\rm{SE}}$ at most not once~/ twice~/ once~/ once.\label{item:ne}
		\item Every hyperedge $E \neq E_b(V, m)$ of $\HH(V, \CR_{\rm{SW}}, m)$ is hit by $S_{\rm{NW}}$~/ $S_{\rm{NE}}$~/ $S_{\rm{SW}}$~/ $S_{\rm{SE}}$ at most once~/ once~/ twice~/ not once.\label{item:sw}
		\item Every hyperedge $E \neq E_b(V, m)$ of $\HH(V, \CR_{\rm{SE}}, m)$ is hit by $S_{\rm{NW}}$~/ $S_{\rm{NE}}$~/ $S_{\rm{SW}}$~/ $S_{\rm{SE}}$ at most once~/ once~/ not once~/ twice.\label{item:se}
		\item The set $E_t(V, m)$ is hit by $S_{\rm{NW}}$~/ $S_{\rm{NE}}$~/ $S_{\rm{SW}}$~/ $S_{\rm{SE}}$ at most once / once / once / once. \label{item:topmost}
		\item The set $E_b(V, m)$ is hit by $S_{\rm{NW}}$~/ $S_{\rm{NE}}$~/ $S_{\rm{SW}}$~/ $S_{\rm{SE}}$ at most once / once / once / once. \label{item:bottommost}
	\end{enumerate}
\end{cor}

Intuitively speaking, \cref{lem:2-shallow-quadrants} allows us to color some points in $V$, in such a way that every north-west quadrant already contains all colors, while other ranges, such as bottomless rectangles or diagonal strips, have most of their points still uncolored.

Now we provide a framework which can then be applied to color various range families. 

\begin{lem}\label{lem:generic-lemma}
    Let $\CR_1, \CR_2$ be shrinkable range families, $f: \NN \to \NN$ be a function, and $s, t \in \NN$ be such that:
    \begin{enumerate}
        \item For every $k \in \NN$, it holds that $m_{\CR_2}(k) \leq f(k)$.
        \item For every point set $V$ and every $m \in \NN$, the hypergraph $\HH(V, \CR_1, m)$ admits a $t$-shallow hitting set $S \subseteq V$ such that every hyperedge of $\HH(V, \CR_2, m)$ it hit at most $s$ times by $S$. \label{item:nice-property}
    \end{enumerate}
    Then for every $k \in \NN$, we have
    \[
        m_{\CR_1 \cup \CR_2}(k) \leq f(k) + k \max(s, t).
    \]
\end{lem}

\begin{proof}
    Let $V$ be a point set and let $k \in \NN$. 
    Let $m = f(k) + k \max(s, t)$.
    We construct a polychromatic coloring of $\HH(V, \CR_1 \cup \CR_2, m)$ with $k$ colors in two steps. First, similarly to \cref{lem:coloring-from-shallow} we iteratively identify $k$ $t$-shallow hitting sets of $\HH(V,\CR_1,m)$ to obtain a $k$-coloring of a subset $V'$ of $V$ that is already polychromatic for $\HH(V,\CR_1,m)$. Second, we show that every hyperedge captured by $\CR_2$ contains at least $f(k)$ (uncolored) points in $V - V'$ and hence, there is a $k$-coloring of $V - V'$ that is polychromatic for $\HH(V,\CR_2,m)$ too. Now we formalize this procedure.
    
    We start with $V_1 = V$, $m_1 = m$ and for $i \in [k]$ repeat the following. Let $S_i$ be a $t$-shallow hitting set of $\HH(V_i, \CR_1, m_i)$ as given by \cref{item:nice-property}. We set $V_{i+1} = V_i \setminus S_i$ and $m_{i+1} = m_i - \max(s, t)$. Note that $m_{i+1} = m - i\max(s, t)$ for all $i \in [k]$ and in particular $m_{k+1} = m -k \max(s,t) = f(k)$.
    
    We claim that for every $i \in [k]$, every hyperedge $E$ of $\HH(V, \CR_1, m)$, and every hyperedge $E'$ of $\HH(V, \CR_2, m)$ the following holds:
    \begin{itemize}
        \item $|S_i \cap E| \geq 1$,
        \item $|V_{i+1} \cap E| \geq m_{i+1}$,
        \item $|V_{i+1} \cap E'| \geq m_{i+1}$.
    \end{itemize}
    We prove the statement by induction on $i$. For $i = 1$ we have $V_1 = V$, $m_1 = m$. Since $S_1$ satisfies \cref{item:nice-property}, $E$ is hit by $S_1$ and both $E$ and $E'$ are hit at most $\max(s, t)$ times by $S_1$, so the claim holds. Now suppose it holds for some $i \geq 1$. Then we know that 
    \[
        |V_{i+1} \cap E| \geq m_{i+1}, \quad |V_{i+1} \cap E'| \geq m_{i+1}.
    \]
    Since $\CR_1$ and $\CR_2$ are shrinkable range families, there exist hyperedges $I \subseteq E$ in $\HH(V_{i+1}, \CR_1, m_{i+1})$ and $I' \subseteq E'$ in $\HH(V_{i+1}, \CR_2, m_{i+1})$.
    For the $t$-shallow hitting set $S_{i+1}$ of $\HH(V_{i+1}, \CR_1, m_{i+1})$ we have
    \[
        1 \leq |S_{i+1} \cap I| \leq t, \quad |S_{i+1} \cap I'| \leq s.
    \]
    First, this implies $|S_{i+1} \cap E| \geq |S_{i+1} \cap I| \geq 1$. Second, we have:
    \[
        |V_{i+2} \cap E| \geq |V_{i+2} \cap I| = |V_{i+1} \cap I| - |S_{i+1} \cap I| \geq m_{i+1} - t \geq m_{i+2}.
    \]
    Similarly, we have:
    \[
        |V_{i+2} \cap E'| \geq |V_{i+2} \cap I'| = |V_{i+1} \cap I'| - |S_{i+1} \cap I'| \geq m_{i+1} - s \geq m_{i+2}.
    \]
    Therefore, the claimed properties hold for $i+1$ too and by induction they hold for every $i \in [k]$.
    First of all, this implies that each of $S_1, \dots, S_k$ is a hitting set of $\HH(V, \CR_1, m)$, where by construction, $S_1, \dots, S_k$ are pairwise disjoint.
    Now we color all points in $S_i$ with color $i$ for every $i \in [k]$. This is a $k$-coloring of $V' = S_1 \cup \cdots \cup S_k \subseteq V$ that is polychromatic for $\HH(V, \CR_1, m)$. 
    For the remaining points $V_{k+1} = V - V'$ we take a polychromatic $k$-coloring of $\HH(V_{k+1},\CR_2,m_{k+1})$, which exists as $m_{k+1} = f(k)$.
    Then every hyperedge $E'$ of $\HH(V, \CR_2, m)$ is polychromatic since $|V_{k+1} \cap E'| \geq m_{k+1}$ and thus, as $\CR_2$ is shrinkable, there exists a hyperedge $I' \subseteq E'$ in $\HH(V_{k+1},\CR_2,m_{k+1})$. 
    Altogether, the arising coloring is a polychromatic coloring of $\HH(V, \CR_1 \cup \CR_2, m)$ and this concludes the proof.
\end{proof}

With \cref{lem:generic-lemma} in place, we can prove the upper bounds for several range families:
\begin{thm}\label{thm:upper-bounds}
    \begin{enumerate}
        \item For the range family $\CR_{\rm{BL}} \cup \CR_{\rm{NW}} \cup \CR_{\rm{NE}}$, we have $m(k) \leq 5k-2$ for all $k$.
        \item For the range family $\CR_{\rm{HS}} \cup \CR_{\rm{VS}} \cup \CR_{\rm{NW}} \cup \CR_{\rm{NE}} \cup \CR_{\rm{SW}} \cup \CR_{\rm{SE}}$, we have $m(k) \leq 10k-1$ for all $k$.
        \item For the range family $\CR_{\rm{HS}} \cup \CR_{\rm{VS}} \cup \CR_{\rm{DS}} \cup \CR_{\rm{NW}} \cup \CR_{\rm{SE}}$, we have $m(k) \leq \lceil 4k \ln k + k \ln 3 \rceil + 4k$ for all $k$.
    \end{enumerate}
\end{thm}

\begin{proof}
    In all cases, the proof combines \cref{cor:shallow-hitting-sets} and \cref{lem:generic-lemma}. To obtain the desired bounds, we only need to choose suitable parameters.
    \begin{enumerate}
        \item By \cref{cor:shallow-hitting-sets}, for every point set $V$ and every $m \in \NN$, there exist subsets $S_{\rm{NW}}, S_{\rm{NE}} \subseteq V$ such that $S_{\rm{NW}} \cup S_{\rm{NE}}$ is a 2-shallow hitting set of $\HH(V, \CR_{\rm{NW}} \cup \CR_{\rm{NE}}, m)$ and every hyperedge of $\HH(V, \CR_{\rm{BL}}, m)$ is hit at most $1+1=2$ times by this set. 
        So we use $s = t = 2$. Further, we set $\CR_1 = \CR_{\rm{NW}} \cup \CR_{\rm{NE}}$, $\CR_2 = \CR_{\rm{BL}}$ and we use $f(k) = 3k-2$ for all $k$. By \cite{ACCCHHKLLMRU13}, we know that for every $k$, we have $m_{\CR_2}(k) \leq f(k)$. So by \cref{lem:generic-lemma}, for every $k$, we have:
        \[
            m(k) \leq (3k - 2) + k\max(2,2) = 5k - 2.
        \]
        
        \item By \cref{cor:shallow-hitting-sets}, for every point set $V$ and every $m \in \NN$, there exist subsets $S_{\rm{NW}}, S_{\rm{NE}},S_{\rm{SW}}, S_{\rm{SE}} \subseteq V$ such that $S_{\rm{NW}} \cup S_{\rm{NE}} \cup S_{\rm{SW}} \cup S_{\rm{SE}}$ is a 4-shallow hitting set of $\HH(V, \CR_{\rm{NW}} \cup \CR_{\rm{NE}} \cup \CR_{\rm{SW}} \cup \CR_{\rm{SE}}, m)$ and every hyperedge of $\HH(V, \CR_{\rm{HS}} \cup \CR_{\rm{VS}}, m)$ is hit at most $2+2+2+2=8$ times by this set. So we use $s = 8, t = 4$. Further, we set $\CR_1 = \CR_{\rm{NW}} \cup \CR_{\rm{NE}} \cup \CR_{\rm{SW}} \cup \CR_{\rm{SE}}$, $\CR_2 = \CR_{\rm{HS}} \cup \CR_{\rm{VS}}$ and we use $f(k) = 2k-1$ for all $k$. By \cite{ACCIKLSST11}, we know that for every $k$ we have $m_{\CR_2}(k) \leq f(k)$. So by \cref{lem:generic-lemma}, for every $k$, we have:
        \[
            m(k) \leq (2k - 1) + k\max(8, 4) = 10k - 1.
        \]
        
        \item By \cref{cor:shallow-hitting-sets}, for every point set $V$ and every $m \in \NN$, there exist subsets $S_{\rm{NW}}, S_{\rm{SE}} \subseteq V$ such that $S_{\rm{NW}} \cup S_{\rm{SE}}$ is a 3-shallow hitting set of $\HH(V, \CR_{\rm{NW}} \cup \CR_{\rm{SE}}, m)$ and every hyperedge of $\HH(V, \CR_{\rm{HS}} \cup \CR_{\rm{VS}} \cup \CR_{\rm{DS}}, m)$ is hit at most $2+2=4$ times by this set. So we use $s = 4, t = 3$. Further, we set $\CR_1 = \CR_{\rm{NW}} \cup \CR_{\rm{SE}}$, $\CR_2 = \CR_{\rm{HS}} \cup \CR_{\rm{VS}} \cup \CR_{\rm{DS}}$ and we use $f(k) = \lceil 4k \ln k + k \ln 3 \rceil$ for all $k$. By \cite{ACCIKLSST11}, we know that for every $k$ we have $m_{\CR_2}(k) \leq f(k)$. So by \cref{lem:generic-lemma}, for every $k$, we have:
        \[
            m(k) \leq \lceil 4k \ln k + k \ln 3 \rceil + k\max(4, 3) = \lceil 4k \ln k + k \ln 3 \rceil + 4k.
        \]
    \end{enumerate}
\end{proof}

\section*{Conclusions}

We have investigated unions of geometric hypergraphs, i.e., for range families $\CR$ that are the union of two range families $\CR_1$ and $\CR_2$, with respect to polychromatic $k$-colorings of the $m$-uniform geometric hypergraphs $\HH(V,\CR,m)$.
We observe the same behavior as for other range families in the literature: Either $m(k) < \infty$ holds for every $k$ or already $m(2) = \infty$.
It remains an interesting open problem to determine whether in general $m(2)<\infty$ always implies $m(k)<\infty$ for all $k$.

In the positive cases, our upper bounds on $m(k)$ are linear in $k$, except when $\CR$ contains strips of three different directions or bottomless and topless rectangles. 
It is worth noting that no range family $\CR$ is known for which $m(2) < \infty$ but $m(k) \in \omega(k)$.
A candidate could be $\CR = \CR_{\rm HS} \cup \CR_{\rm VS} \cup \CR_{\rm DS}$ or $\CR = \CR_{\rm BL} \cup \CR_{\rm TL}$.

Finally, we suggest a further investigation of shallow hitting sets in these geometric hypergraphs.
To the best of our knowledge, it might be true that their existence is equivalent to $m(k)$ being linear in $k$.
In particular, do bottomless rectangles (for which it is known that $m(k) \in O(k)$ \cite{ACCCHHKLLMRU13}) allow for shallow hitting sets?
And do octants in 3D (for which shallow hitting sets are known not to exist~\cite{CKMPUV20}) have $m(k) \in O(k)$?

\bibliographystyle{abbrvnat}
\bibliography{lit}

\end{document}